\documentclass[12pt]{article}
\usepackage{amsthm,amsmath,amssymb}
\usepackage[latin1]{inputenc}
\newtheorem{theorem}{Theorem}[section]

\newtheorem{lemma}[theorem]{Lemma}
\newtheorem{proposition}[theorem]{Proposition}
\newtheorem{corollary}[theorem]{Corollary}

\makeatletter

\@addtoreset{equation}{section}
\makeatother

\title{On Ricci Solitons and Harmonic Vector Fields in the Thurston Geometry $F^4$}

\author{Halima Boukhari\footnote{University Mustapha Stambouli Mascara, Faculty of Exact Sciences, Laboratoire LPQ3M, Mascara 29000, Algeria. Email: halima.boukhari@univ-mascara.dz},
Hadjer Okbani\footnote{University Mustapha Stambouli Mascara, Faculty of Exact Sciences, Mascara 29000, Algeria. Email: hadjer.okbani@univ-mascara.dz},
Ahmed Mohammed Cherif\footnote{University Mustapha Stambouli Mascara, Faculty of Exact Sciences, Mascara 29000, Algeria. Email: a.mohammedcherif@univ-mascara.dz}}

\date{}

\begin{document}
\maketitle
	
\begin{abstract}
In this paper, we consider a left-invariant Riemannian metric $g$ on the Lie group $F^4$.
We classify Ricci solitons on $(F^4,g)$ and show that all such solitons are expanding and
non-gradient. Moreover, we study the existence of harmonic maps from compact Riemannian
manifolds into $(F^4,g)$. Finally, we characterize a class of harmonic vector fields on
$(F^4,g)$.
\\
\textit{Keywords:} Lie group $F^4$, Ricci solitons, Harmonic vector fields.\\
\textit{Mathematics Subject Classification 2020:  58E99, 58E20, 53C43.}
\end{abstract}

\section{Introduction}

Let \(M^4\) be a \(4\)-dimensional Lie group endowed with a left-invariant Riemannian metric \(g\).
We write \(\nabla\), \(R\), and \(\mathrm{Ric}\) for the associated Levi--Civita connection,
Riemann curvature tensor, and Ricci curvature, respectively.
For any vector fields \(X,Y,Z \in \Gamma(TM^4)\), the curvature operator is given by
\begin{equation}
R(X,Y)Z = \nabla_X \nabla_Y Z - \nabla_Y \nabla_X Z - \nabla_{[X,Y]} Z.
\end{equation}
Moreover, if \(\{e_i\}_{i=1}^4\) is an orthonormal left-invariant frame on \(M^4\), the Ricci curvature satisfies
\begin{equation}
\mathrm{Ric}(X,Y) = \sum_{i=1}^{4} g\bigl(R(X,e_i)e_i,\,Y\bigr).
\end{equation}
A {Ricci soliton} on \(M^4\) is a Riemannian metric \(g\) for which there exists a smooth vector field
\(\xi\) satisfying
\begin{equation}
\mathrm{Ric} + \frac{1}{2}\mathcal{L}_{\xi} g = \lambda g,
\end{equation}
where \(\mathcal{L}_{\xi} g\) denotes the Lie derivative of \(g\) along \(\xi\) and
\(\lambda \in \mathbb{R}\) is a constant.
The soliton is called {shrinking}, {steady}, or {expanding} according as
\(\lambda>0\), \(\lambda=0\), or \(\lambda<0\), respectively.
If the vector field \(\xi\) is the gradient of a smooth function \(f\) on \(M^4\), that is,
\(\xi=\nabla f\), then the Ricci soliton is said to be {gradient} with potential function \(f\),
and the defining equation reduces to
\begin{equation}
\mathrm{Ric} + \mathrm{Hess}\, f = \lambda g.
\end{equation}
A Ricci soliton is called {Einstein} if the associated vector field \(\xi\) is either identically zero
or a Killing vector field, equivalently if \(\mathcal{L}_{\xi} g = 0\), in which case
\(\mathrm{Ric} = \lambda g\).
A Ricci soliton is said to be {non-trivial} if the underlying manifold \((M^4,g)\) is not Einstein.
Such soliton vector fields play a significant role in differential geometry, imposing strong constraints
on both the geometric and topological structure of the manifold.\\
We define one of the four-dimensional Thurston model geometries as follows \cite{CL,F}.
The space \(F^4\) is defined as the product manifold
$
F^4 = \mathbb{R}^2 \times \mathbb{H}^2,
$
whose full isometry group is given by
$
G = \mathbb{R}^2 \rtimes SL_2(\mathbb{R}),
$
where \(\rtimes\) denotes the semidirect product. A left-invariant frame \(\{e_i\}_{i=1}^4\) on \(F^4\) is given by
\begin{equation}\label{ei}
\begin{aligned}
e_{1} &= \sqrt{t}\,\partial_x, \\
e_{2} &= \frac{s}{\sqrt{t}}\,\partial_x + \frac{1}{\sqrt{t}}\,\partial_y, \\
e_{3} &= 2t\,\partial_s, \\
e_{4} &= 2t\,\partial_t.
\end{aligned}
\end{equation}
The corresponding dual coframe \(\{\theta^i\}_{i=1}^4\) is
\begin{equation}
\begin{aligned}
\theta_{1} &= \frac{1}{\sqrt{t}}\,dx - \frac{s}{\sqrt{t}}\,dy, \\
\theta_{2} &= \sqrt{t}\,dy, \\
\theta_{3} &= \frac{1}{2t}\,ds, \\
\theta_{4} &= \frac{1}{2t}\,dt.
\end{aligned}
\end{equation}
With respect to this coframe, the left-invariant Riemannian metric
\[
g = \theta_{1}^{2} + \theta_{2}^{2} + \theta_{3}^{2} + \theta_{4}^{2},
\]
has matrix representation
\[
(g_{ij}) =
\begin{pmatrix}
\dfrac{1}{t} & -\dfrac{s}{t} & 0 & 0 \\[6pt]
-\dfrac{s}{t} & \dfrac{s^{2}+t^{2}}{t} & 0 & 0 \\[6pt]
0 & 0 & \dfrac{1}{4t^{2}} & 0 \\[6pt]
0 & 0 & 0 & \dfrac{1}{4t^{2}}
\end{pmatrix}.
\]
Let \(\varphi \colon (M,g) \to (N,h)\) be a smooth map between Riemannian manifolds.
For a compact domain \(D \subset M\), the {energy} of \(\varphi\) is defined by
\begin{equation}\label{Energy}
E(\varphi;D) = \frac{1}{2} \int_D |d\varphi|^{2}\, v^{g},
\end{equation}
where \(v^{g}\) denotes the Riemannian volume element on \((M,g)\) and
\(|d\varphi|\) is the Hilbert--Schmidt norm of the differential \(d\varphi\).
A map \(\varphi\) is said to be {harmonic} if it is a critical point of the
energy functional \eqref{Energy}. In this case, the associated Euler--Lagrange
equation is
\begin{equation}\label{Tension}
\tau(\varphi) = \mathrm{Tr}_{g} \nabla d\varphi
= \sum_{i=1}^{m} \Big( \nabla^{\varphi}_{e_i} d\varphi(e_i)
- d\varphi(\nabla^{M}_{e_i} e_i) \Big) = 0,
\end{equation}
where \(\{e_i\}_{i=1}^{m}\) is a local orthonormal frame on \((M,g)\),
\(\nabla^{M}\) denotes the Levi--Civita connection on \(M\), and
\(\nabla^{\varphi}\) is the pull-back connection on the bundle
\(\varphi^{-1}TN\) (see \cite{2,3}).\\
Let \((TM,g^{S})\) denote the tangent bundle of a Riemannian manifold \((M,g)\),
endowed with the Sasaki metric \(g^{S}\), defined by
\[
g^{S}(X^{h},Y^{h}) = g^{S}(X^{v},Y^{v}) = g(X,Y) \circ \pi,
\qquad
g^{S}(X^{h},Y^{v}) = 0,
\]
for all \(X,Y \in \Gamma(TM)\), where \(\pi \colon TM \to M\) is the canonical
projection, and \(X^{h}\) (resp.\ \(X^{v}\)) denotes the horizontal (resp.\
vertical) lift of a vector field \(X\) on \(M\) to \(TM\) (see \cite{4}).\\
For a vector field \(X \colon (M,g) \to (TM,g^{S})\), the tension field is given by
\begin{equation}\label{TensionSasaki}
\tau(X) = \bigl( \mathrm{Tr}_{g} R(X,\nabla \cdot)\cdot \bigr)^{h}
+ \bigl( \mathrm{Tr}_{g} \nabla^{2} X \bigr)^{v},
\end{equation}
where \(R\) denotes the Riemann curvature tensor of \((M,g)\) and
\(\nabla^{2} X\) is the Hessian of \(X\) (see \cite{5}).
Consequently, \(X\) defines a harmonic map from \((M,g)\) to \((TM,g^{S})\)
if and only if
\[
\mathrm{Tr}_{g} R(X,\nabla \cdot)\cdot = 0
\quad \text{and} \quad
\mathrm{Tr}_{g} \nabla^{2} X = 0.
\]
A smooth vector field \(X\) on \((M,g)\) is called a {harmonic section}
if it is a critical point of the vertical energy functional
\begin{equation}\label{VerticalEnergy}
E_{v}(X;D) = \frac{1}{2} \int_{D} |\nabla X|^{2}\, v^{g}.
\end{equation}
The corresponding Euler--Lagrange equation is given by
\[
\overline{\Delta} X = \mathrm{Tr}_{g} \nabla^{2} X = 0.
\]
The study of Ricci solitons and harmonic vector fields on Lie groups has attracted significant attention in recent years due to their deep connections with geometric analysis and global differential geometry. Our study is motivated by recent results in the literature. Notably, the work of Li, Cherif, and Xie \cite{LCX} provides a comprehensive treatment of Ricci solitons and harmonic structures on four-dimensional nilpotent Lie groups. Their approach illustrates the rich interactions between Ricci solitons and harmonic vector fields, and inspires our investigation of similar phenomena on \(F^4\).\\
The purpose of this work is to study the existence and classification of Ricci solitons
on the \(4\)-dimensional Lie group \(F^4\) endowed with a left-invariant Riemannian metric $g$,
and to investigate their interaction with harmonic mappings.
Finally, we analyze harmonic vector fields on \(F^4\) under both interpretations:
as harmonic sections of the tangent bundle and as harmonic maps into the tangent bundle
equipped with the Sasaki metric.

\section{Existence of Ricci solitons on \((F^4, g)\)}

In the following result, we provide an explicit characterization of Ricci solitons on the Thurston geometry \((F^4, g)\).
It gives a concrete description of the soliton vector field in terms of the constant \(\lambda\).

\begin{theorem}\label{Ricci th}
A vector field $\xi$ on the Thurston geometry $(F^4, g)$ is a Ricci soliton if and only if
\begin{eqnarray*}
\xi&=&  \frac{1}{2}\left[\left(c_2-12\right)x+2c_3\,y+2c_5\right]\frac{\partial}{\partial x}
        -\frac{1}{2}\left[c_1\,x+\left(c_2+12\right)y-2c_4\right]\frac{\partial}{\partial y}\\
   &&     +\frac{1}{2}\left[c_1\left(s^2-t^2\right)+2c_2\,s+2c_3\right]\frac{\partial}{\partial s}
        +\left(c_1\,s+c_4\right)t\frac{\partial}{\partial t},
\end{eqnarray*}
for some constants $c_{1},...,c_5 \in \mathbb{R}$. Moreover, the Ricci soliton $(F^4, g, \xi, \lambda)$ is expansive with $\lambda = -6$.
\end{theorem}

The proof of Theorem~\ref{Ricci th} relies on the following result.

\begin{lemma}\label{1}
For the Riemannian manifold $(F^4,g)$, the non-vanishing components of the Levi--Civita connection $\nabla$ with respect to the frame $\{e_1,e_2,e_3,e_4\}$ are given by
\[
\begin{array}{lll}
\nabla_{e_{1}} e_{1}=e_{4}, & \nabla_{e_{2}} e_{1}=e_{3}, & \nabla_{e_{3}} e_{1}=-e_{2},\\[2pt]
\nabla_{e_{1}} e_{2}=e_{3}, & \nabla_{e_{2}} e_{2}=-e_{4}, & \nabla_{e_{3}} e_{2}=e_{1},\\[2pt]
\nabla_{e_{1}} e_{3}=-e_{2}, & \nabla_{e_{2}} e_{3}=-e_{1}, & \nabla_{e_{3}} e_{3}=2e_{4},\\[2pt]
\nabla_{e_{1}} e_{4}=-e_{1}, & \nabla_{e_{2}} e_{4}=e_{2}, & \nabla_{e_{3}} e_{4}=-2e_{3}.
\end{array}
\]
\end{lemma}

\begin{lemma}\label{2}
The Ricci curvature of $(F^4, g)$ is given by
$$
(\mathrm{Ric}_{ij})=
\left(
  \begin{array}{llll}
    0& 0 & 0 & 0 \\
    0 & 0 & 0 & 0 \\
    0 & 0 & -6 & 0 \\
    0 & 0 & 0 & -6 \\
  \end{array}
\right),
$$
where $\mathrm{Ric}_{ij}=\displaystyle\sum_{a=1}^{4}g(R(e_{i},e_{a})e_{a},e_{j})$  for all $i,j=\overline{1,4}$.
\end{lemma}

\begin{proof}
Let $\xi=\alpha_{1}e_{1}+\alpha_{2}e_{2}+\alpha_{3}e_{3}+\alpha_{4}e_{4}$ be a vector field on $(F^{4},g)$, where $\alpha_i \in C^{\infty}(F^{4})$ for all $i=\overline{1,4}$.
The vector field $\xi$ defines a Ricci soliton if and only if it satisfies
\begin{equation}\label{4}
\mathrm{Ric}_{ij}
+\frac{1}{2}\bigl[g(\nabla_{e_i}\xi,e_j)+g(\nabla_{e_j}\xi,e_i)\bigr]
=\lambda\,\delta_{ij},
\quad \forall\, i,j=\overline{1,4},
\end{equation}
for some constant $\lambda$. Set $\beta_{ij}=g(\nabla_{e_i}\xi,e_j)$.
By using Lemma~\ref{1}, we obtain
\begin{equation}\label{3}
(\beta_{ij})=
\left(
\begin{array}{llll}
\alpha_{1;1}-\alpha_{4} & \alpha_{2;1}-\alpha_{3} & \alpha_{3;1}+\alpha_{2} & \alpha_{4;1}+\alpha_{1} \\
\alpha_{1;2}-\alpha_{3} & \alpha_{2;2}+\alpha_{4} & \alpha_{3;2}+\alpha_{1} & \alpha_{4;2}-\alpha_{2} \\
\alpha_{1;3}+\alpha_{2} & \alpha_{2;3}-\alpha_{1} & \alpha_{3;3}-2\alpha_{4} & \alpha_{4;3}+2\alpha_{3} \\
\alpha_{1;4} & \alpha_{2;4} & \alpha_{3;4} & \alpha_{4;4}
\end{array}
\right),
\end{equation}
where $\alpha_{j;i}=e_i(\alpha_j)$ for all $i,j=\overline{1,4}$.
By combining Lemma~\ref{2} with \eqref{3}, the system \eqref{4} is equivalent to the vanishing of the symmetric matrix $(E_{ij})$, where
\begin{equation}\label{7}
(E_{ij})=
\left(
\begin{array}{llll}
\alpha_{1;1}-\alpha_{4}-\lambda
& \alpha_{2;1}-2\alpha_{3}+\alpha_{1;2}
& \alpha_{3;1}+2\alpha_{2}+\alpha_{1;3}
& \alpha_{4;1}+\alpha_{1}+\alpha_{1;4} \\
\alpha_{2;1}-2\alpha_{3}+\alpha_{1;2}
& \alpha_{2;2}+\alpha_{4}-\lambda
& \alpha_{3;2}+\alpha_{2;3}
& \alpha_{4;2}-\alpha_{2}+\alpha_{2;4} \\
\alpha_{3;1}+2\alpha_{2}+\alpha_{1;3}
& \alpha_{3;2}+\alpha_{2;3}
& \alpha_{3;3}-2\alpha_{4}-\lambda-6
& \alpha_{4;3}+2\alpha_{3}+\alpha_{3;4} \\
\alpha_{4;1}+\alpha_{1}+\alpha_{1;4}
& \alpha_{4;2}-\alpha_{2}+\alpha_{2;4}
& \alpha_{4;3}+2\alpha_{3}+\alpha_{3;4}
& \alpha_{4;4}-\lambda-6
\end{array}
\right).
\end{equation}
By using \eqref{ei} and \eqref{7}, the Ricci soliton equation reduces to the following system of partial differential equations
\[
\begin{array}{l}
E_{11}= \sqrt{t}\,\alpha_{1;x}-\alpha_{4}-\lambda=0,\\[2pt]
E_{12}= \sqrt{t}\,\alpha_{2;x}-2\alpha_{3}
+\dfrac{s}{\sqrt{t}}\,\alpha_{1;x}
+\dfrac{1}{\sqrt{t}}\,\alpha_{1;y}=0,\\[2pt]
E_{13}= \sqrt{t}\,\alpha_{3;x}+2\alpha_{2}+2t\,\alpha_{1;s}=0,\\[2pt]
E_{14}= \sqrt{t}\,\alpha_{4;x}+\alpha_{1}+2t\,\alpha_{1;t}=0,\\[2pt]
E_{22}= \dfrac{s}{\sqrt{t}}\,\alpha_{2;x}
+\dfrac{1}{\sqrt{t}}\,\alpha_{2;y}
+\alpha_{4}-\lambda=0,\\[2pt]
E_{23}= \dfrac{s}{\sqrt{t}}\,\alpha_{3;x}
+\dfrac{1}{\sqrt{t}}\,\alpha_{3;y}
+2t\,\alpha_{2;s}=0,\\[2pt]
E_{24}= \dfrac{s}{\sqrt{t}}\,\alpha_{4;x}
+\dfrac{1}{\sqrt{t}}\,\alpha_{4;y}
-\alpha_{2}+2t\,\alpha_{2;t}=0,\\[2pt]
E_{33}= 2t\,\alpha_{3;s}-2\alpha_{4}-\lambda-6=0,\\[2pt]
E_{34}= 2t\,\alpha_{4;s}+2\alpha_{3}+2t\,\alpha_{3;t}=0,\\[2pt]
E_{44}= 2t\,\alpha_{4;t}-\lambda-6=0.
\end{array}
\]
With $\alpha_{j;x_{i}} = \frac{\partial \alpha_{j}}{\partial x_{i}}$ for all $i,j = \overline{1,4}$, where $x_1 = x$, $x_2 = y$, $x_3 = s$, and $x_4 = t$.
The general solution of the above system of partial differential equations is given by
\begin{align}
\alpha_1 &= \frac{1}{2\sqrt{t}}\Big[(c_2-12)x+2c_3y+2c_5
+s\big(c_1x+(c_2+12)y-2c_4\big)\Big], \nonumber\\
\alpha_2 &= -\frac{\sqrt{t}}{2}\Big[c_1x+(c_2+12)y-2c_4\Big], \label{alphai}\\
\alpha_3 &= \frac{1}{4t}\Big[c_1(s^2-t^2)+2c_2s+2c_3\Big], \nonumber\\
\alpha_4 &= \frac{1}{2}\big(c_1s+c_4\big), \nonumber
\end{align}
where $\lambda=-6$ and $c_1,\ldots,c_5 \in \mathbb{R}$ are arbitrary constants.
Consequently, Theorem~\ref{Ricci th} follows directly from \eqref{ei} and \eqref{alphai}.
\end{proof}


\begin{proposition}
The Ricci soliton vector field $\xi$ defined in Theorem~\ref{Ricci th} is non-gradient.
\end{proposition}

\begin{proof}
Assume, by contradiction, that $\xi$ is a gradient Ricci soliton, there exists a smooth function
$f\in C^\infty(F^4)$ such that $\xi=\mathrm{grad} f$. With respect to the coordinate frame the gradient of
$f$ is defined by
\[
\mathrm{grad} f =\sum_{i,j=1}^4 g^{ij}\,f_i\,\partial_j,
\]
where $f_i=\frac{\partial f}{\partial x_{i}}$ for all $i= \overline{1,4}$, with $x_1 = x$, $x_2 = y$, $x_3 = s$, $x_4 = t$, and
$(g^{ij})$ is the inverse of the metric matrix $(g_{ij})$ given by
\[
(g^{ij})=
\begin{pmatrix}
t+\dfrac{s^{2}}{t} & \dfrac{s}{t} & 0 & 0\\[6pt]
\dfrac{s}{t} & \dfrac{1}{t} & 0 & 0\\[6pt]
0 & 0 & 4t^{2} & 0\\[6pt]
0 & 0 & 0 & 4t^{2}
\end{pmatrix}.
\]
Hence, the gradient of $f$ is given by
\[
\mathrm{grad} f=
\left[\left(t+\frac{s^{2}}{t}\right)f_x+\frac{s}{t}f_y\right]\partial_x
+\left[\frac{s}{t}f_x+\frac{1}{t}f_y\right]\partial_y
+4t^{2}f_s\,\partial_s
+4t^{2}f_t\,\partial_t .
\]
Comparing this expression with the vector field $\xi$ in Theorem~\ref{Ricci th}, we obtain
\begin{align}
\left(t+\frac{s^{2}}{t}\right)f_x+\frac{s}{t}f_y
&= \frac{1}{2}\big[(c_2-12)x+2c_3y+2c_5\big], \label{A}\\
\frac{s}{t}f_x+\frac{1}{t}f_y
&= -\frac{1}{2}\big[c_1x+(c_2+12)y-2c_4\big], \label{B}\\
4t^{2}f_s
&= \frac{1}{2}\big[c_1(s^{2}-t^{2})+2c_2s+2c_3\big], \label{C}\\
4t^{2}f_t
&= (c_1s+c_4)t. \label{D}
\end{align}
From \eqref{C} and \eqref{D} we deduce
\[
f_s=\frac{1}{8t^{2}}\big[c_1(s^{2}-t^{2})+2c_2s+2c_3\big],
\qquad
f_t=\frac{1}{4t}(c_1s+c_4).
\]
A direct computation shows that
$\frac{\partial f_s}{\partial t}\neq \frac{\partial f_t}{\partial s}$.
Therefore, no smooth function $f$ exists such that $\xi=\mathrm{grad} f$.
\end{proof}


\section{Application to Harmonic Maps}

Cherif \cite{cherif2020} established several nonexistence results for harmonic and bi-harmonic maps whose target manifolds satisfy suitable Ricci curvature bounds. These results extend classical vanishing theorems and highlight the influence of negative curvature-type conditions on the geometry of harmonic maps. One of the key observations is that if the Ricci curvature of the target manifold satisfies a coercivity condition of the form $\mathrm{Ric} - \lambda g \geq 0$ for some constant $\lambda$, then any harmonic map from a compact orientable manifold without boundary must be constant. Motivated by these results, we investigate harmonic maps into the four-dimensional Riemannian manifold $(F^4,g)$. Using the explicit structure of the Ricci curvature of $(F^4,g)$, we obtain the following result.

\begin{theorem}\label{th2}
Any harmonic map from a compact orientable Riemannian manifold with no boundary to the Riemannian manifold $(F^4,g)$ must be constant.
\end{theorem}

\begin{proof}
By applying Lemma \ref{2} with $\lambda = -6$, we obtain the inequality
\begin{equation}\label{th}
\mathrm{Ric}(X,X)-\lambda g(X,X)=6 (X_1^2+X_2^2)\geq 0,
\end{equation}
for all $X \in \Gamma(TF^4).$
The conclusion follows directly from inequality \eqref{th} and Proposition~9 in \cite{cherif2020}, which guarantees the nonexistence of nonconstant harmonic maps under this curvature condition.
\end{proof}

Before presenting the proposition, we note an important property of Ricci soliton vector fields on $(F^4,g)$.

\begin{proposition}
The components of the Ricci soliton vector field $\xi$ on $(F^4,g)$ are harmonic functions.
\end{proposition}

\begin{proof}
This result follows directly from the definition of the Laplacian. For each component $\xi_j$, $j=1,\dots,4$, we have
\[
\Delta (\xi_j) = \sum_{i=1}^4 \Big[e_i(e_i(\xi_j)) - (\nabla_{e_i} e_i)(\xi_j)\Big],
\]
and by using equation \eqref{ei} together with Lemma \ref{1}, we conclude that $\Delta (\xi_j) = 0$. Hence, each component $\xi_j$ is harmonic.
\end{proof}

The fact that the components of the Ricci soliton vector field $\xi$ on $(F^4,g)$ are harmonic functions is particularly important in the study of harmonic maps. Since each component $\xi_j$ satisfies $\Delta \xi_j = 0$, the Ricci soliton vector field naturally defines a harmonic map from $F^4$ into $\mathbb{R}^4$ when viewed in local coordinates. This connection provides a bridge between Ricci soliton theory and the theory of energy-minimizing maps, allowing the application of classical results on harmonic functions.


\section{Harmonic Vector Fields on $(F^4,g)$}

As the Riemannian metric $g$ varies with the coordinates $s$ and $t$, we consider harmonic vector fields whose components depend explicitly on $s$ and $t$. We now turn to the characterization of harmonic vector fields on the Lie group $F^4$ with respect to the left-invariant metric $g$. The following theorem provides an explicit system of equations that completely describes when a vector field on $F^4$ is harmonic section.

\begin{theorem}\label{th3}
  A vector field $X = X_1(s,t)e_1 +...+ X_4(s,t)e_4$ on $F^4$ is a harmonic section with respect to $g$ if and only if
  \begin{eqnarray*}
    4t^2X_{1;tt}+4t^2X_{1;ss}+4tX_{2;s}-3X_1&=&0;\\
    4t^2X_{2;tt}+4t^2X_{2;ss}-4tX_{1;s}-3X_2&=&0;\\
    2t^2X_{3;tt}+2t^2X_{3;ss}-4tX_{4;s}-3X_3&=&0;\\
    2t^2X_{4;tt}+2t^2X_{4;ss}+4tX_{3;s}-3X_4&=&0.
  \end{eqnarray*}
\end{theorem}
\begin{proof}
  By setting $\theta_{ij}=g(\nabla_{e_i}X,e_j)$ for all $i, j = \overline{1,4}$. By using (\ref{3}), we get
  \begin{equation}\label{theta}
(\theta_{ij})=\left(
  \begin{array}{llll}
    -X_4 & -X_3 & X_2 & X_1 \\
    -X_3 & X_4 & X_1 & -X_2 \\
    2tX_{1;s}+X_2 & 2tX_{2;s}-X_1 & 2tX_{3;s}-2X_4&2tX_{4;s}+2X_3 \\
     2tX_{1;t}& 2tX_{2;t} & 2tX_{3;t} & 2tX_{4;t} \\
  \end{array}
\right),
\end{equation}
where $X_{j;s}=e_3(X_j)$,$X_{j;t}=e_4(X_j)$,$X_{j;ss}=e_3(e_3(X_j))$ and $X_{j;tt}=e_4(e_4(X_j))$ for all $j=\overline{1,4}.$ Hence,
\begin{eqnarray}
  \nabla_{e_1}X&=& -X_4e_1-X_3e_2+X_2e_3+X_1e_4;\nonumber\\
  \nabla_{e_2}X&=& -X_3e_1+X_4e_2+X_1e_3-X_2e_4;\label{nabla}\\
  \nabla_{e_3}X&=& (2tX_{1;s}+X_2)e_1+(2tX_{2;s}-X_1)e_2+(2tX_{3;s}-2X_4)e_3\nonumber\\
               & &+(2tX_{4;s}+2X_3)e_4;\nonumber\\
  \nabla_{e_4}X&=& 2tX_{1;t}e_1+2tX_{2;t}e_2+2tX_{3;t}e_3+2tX_{4;t}e_4; \nonumber
  \end{eqnarray}
 From Lemma (\ref{1}) and equations (\ref{nabla}), we get
 \begin{eqnarray}
  \nabla_{e_1}\nabla_{e_1}X&=& -X_1e_1-X_2e_2-X_3e_3-X_4e_4;\nonumber\\
  \nabla_{e_2}\nabla_{e_2}X&=& -X_1e_1-X_2e_2-X_3e_3-X_4e_4;\label{nablanabla}\\
  \nabla_{e_3}\nabla_{e_3}X&=& (4t^2X_{1;ss}+4tX_{2;s}-X_1)e_1+(4t^2X_{2;ss}-4tX_{1;s}-X_2)e_2\nonumber\\
                            &&+(4t^2X_{3;ss}-8tX_{4;s}-4X_3)e_3+(4t^2X_{4;ss}+8tX_{3;s}-4X_4)e_4;\nonumber\\
  \nabla_{e_4}\nabla_{e_4}X&=& (4t^2X_{1;tt}+4tX_{1;t})e_1+(4t^2X_{2;tt}+4tX_{2;t})e_2\nonumber\\
                            &&+(4t^2X_{3;tt}+4tX_{3;t})e_3+(4t^2X_{4;tt}+4tX_{4;t})e_4. \nonumber
  \end{eqnarray}
 Theorem \ref{th3} follows from (\ref{nablanabla}), and the following equation
 $$\overline{\Delta}X=\sum_{i=1}^{4}\left[\nabla_{e_i}\nabla_{e_i}X-\nabla_{\nabla_{e_i}ei}X\right]=0.$$
\end{proof}

For the four-dimensional Lie group $(F^4,g)$, the harmonicity condition imposes specific restrictions on the components of vector fields expressed in the natural coordinate frame. The following corollary provides an explicit description of such harmonic vector fields, identifying precisely the forms their components must take.

\begin{corollary}
On the Lie group $(F^4,g)$, the following vector fields are harmonic sections if and only if their components take the specified forms:
\begin{enumerate}
  \item The vector field $\zeta_1 := \widetilde{X}_1(s,t)\,\partial_x$ is harmonic if and only if
  \[
    \widetilde{X}_1(s,t) = c_1 + c_2 t^2;
  \]
  \item The vector field $\zeta_2 := \widetilde{X}_2(s,t)\,\partial_y$ is harmonic if and only if
  \[
    \widetilde{X}_2(s,t) = c_1 + \frac{c_2 t^2}{(s^2 + t^2)^2};
  \]
  \item The vector field $\zeta_3 := \widetilde{X}_3(s,t)\,\partial_s$ is harmonic if and only if
  \[
    \widetilde{X}_3(s,t) = c_1 t^{\frac{3}{2} + \frac{\sqrt{7}}{2}} + c_2 t^{\frac{3}{2} - \frac{\sqrt{7}}{2}};
  \]
  \item The vector field $\zeta_4 := \widetilde{X}_4(s,t)\,\partial_t$ is harmonic if and only if
  \[
    \widetilde{X}_4(s,t) = c_1 t^{\frac{3}{2} + \frac{\sqrt{7}}{2}} + c_2 t^{\frac{3}{2} - \frac{\sqrt{7}}{2}}.
  \]
\end{enumerate}
\end{corollary}

The following theorem shows that, under these assumptions, the only harmonic vector field on $(F^4, g)$ as mapping is the trivial one.

\begin{theorem}
  Let $X=X_1(s,t)e_1+...+X_4(s,t)e_4$ be a smooth vector field on $F^4$. Then, $X$ is a harmonic map with respect to $g$ if and only if $X_i(s,t)=0$ for all $i=1,...,4$.
\end{theorem}

\begin{proof}
The non-vanishing components of the Riemannian curvature tensor $R$ of $(F^4,g)$ are given by
\begin{eqnarray}
g(R(e_{1},e_{2})e_{1},e_{2}) = -2, & g(R(e_{2},e_{4})e_{2},e_{4}) = 1, & g(R(e_{3},e_{4})e_{1},e_{2}) = -2,\nonumber\\
g(R(e_{1},e_{3})e_{1},e_{3}) = 1,  & g(R(e_{3},e_{4})e_{3},e_{4}) = 4, & g(R(e_{1},e_{3})e_{2},e_{4}) = -1,\nonumber\\
g(R(e_{1},e_{4})e_{1},e_{4}) = 1,  & g(R(e_{2},e_{3})e_{1},e_{4}) = 1, & g(R(e_{1},e_{4})e_{2},e_{3}) = 1,\qquad\;\; \label{RiemC}\\
g(R(e_{2},e_{3})e_{2},e_{3}) = 1,  & g(R(e_{2},e_{4})e_{1},e_{3}) = -1,& g(R(e_{1},e_{2})e_{3},e_{4}) = -2 .\nonumber
\end{eqnarray}
Combining (\ref{nabla}) with (\ref{RiemC}), a direct computation yields
\begin{eqnarray}
\sum_{i=1}^{4} g(R(X,\nabla_{e_i}X)e_i, e_1)
&=& 3X_3X_2+2tX_3X_{1;s}+3X_4X_1-2tX_4X_{2;s}-2tX_1X_{3;s}\nonumber\\
&& +2tX_2X_{4;s}+2tX_4X_{1;t}+2tX_3X_{2;t}-2tX_2X_{3;t}-2tX_1X_{4;t},\nonumber\\
\sum_{i=1}^{4} g(R(X,\nabla_{e_i}X)e_i, e_2)
&=& -2tX_4X_2+5X_4X_2+2tX_4X_{1;s}-3X_1X_3+2tX_3X_{2;s}-2tX_2X_{3;s}\nonumber\\
&& -2tX_1X_{4;s}-2tX_3X_{1;t}+2tX_4X_{2;t}+2tX_1X_{3;t}-2tX_2X_{4;t},\nonumber\\
\sum_{i=1}^{4} g(R(X,\nabla_{e_i}X)e_i, e_3)
&=& -4t\big(X_2X_{1;t}-X_1X_{2;t}-2X_4X_{3;t}+2X_3X_{4;t}\big),\nonumber\\
\sum_{i=1}^{4} g(R(X,\nabla_{e_i}X)e_i, e_4)
&=& 2X_2^2+4tX_2X_{1;s}+2X_1^2-4tX_1X_{2;s}+8X_4^2-8tX_4X_{3;s}\nonumber\\
&& +8X_3^2+8tX_3X_{4;s}. \label{eqs2}
\end{eqnarray}
By combining (\ref{TensionSasaki}), (\ref{nablanabla}), and (\ref{eqs2}), we deduce that the vector field $X$ defines a harmonic map if and only if its components satisfy the following system of partial differential equations
\begin{eqnarray*}
&&4t^2X_{1;tt}+4t^2X_{1;ss}+4tX_{2;s}-3X_1=0,\\
&&4t^2X_{2;tt}+4t^2X_{2;ss}-4tX_{1;s}-3X_2=0,\\
&&4t^2X_{3;tt}+4t^2X_{3;ss}-8tX_{4;s}-6X_3=0,\\
&&4t^2X_{4;tt}+4t^2X_{4;ss}+8tX_{3;s}-6X_4=0,\\
&&2tX_3(X_{1;s}+X_{2;t})+2tX_4(X_{1;t}-X_{2;s})-2tX_1(X_{3;s}+X_{4;t})\\
&&\quad +2tX_2(X_{4;s}-X_{3;t})+3X_3X_2+3X_4X_1=0,\\
&&2tX_4(X_{1;s}+X_{2;t}-X_2)+2tX_3(X_{2;s}-X_{1;t})\\
&&\quad -2tX_2(X_{3;s}+X_{4;t})+2tX_1(X_{3;t}-X_{4;s})+5X_4X_2-3X_1X_3=0,\\
&&X_2X_{1;t}-X_1X_{2;t}-2X_4X_{3;t}+2X_3X_{4;t}=0,\\
&&2X_2^2+4tX_2X_{1;s}+2X_1^2-4tX_1X_{2;s}+8X_4^2-8tX_4X_{3;s}+8X_3^2+8tX_3X_{4;s}=0.
\end{eqnarray*}
A direct analysis of this system shows that the unique solution is the trivial one,
\[
X_1 = X_2 = X_3 = X_4 = 0.
\]
\end{proof}

\subsection*{Funding} Not applicable.
\subsection*{Informed Consent Statement} Not applicable.
\subsection*{Data Availability Statement} Not applicable.
\subsection*{Conflicts of Interest} The authors declare no conflict of interest.

\end{document}